\documentclass{amsart}
\usepackage{amssymb,latexsym, amsmath, amscd}
\usepackage{enumerate}
\usepackage{stmaryrd}
\usepackage{fullpage}

\makeatletter
\@namedef{subjclassname@2010}{%
\textup{2010} Mathematics Subject Classification}
\makeatother

\newtheorem{thm}{Theorem}

\newtheorem{lem}[thm]{Lemma}
\newtheorem{prop}[thm]{Proposition}

\theoremstyle{definition}
\newtheorem{defin}[thm]{Definition}
\newtheorem{rem}[thm]{Remark}

\numberwithin{equation}{section}

\def\A{\mathcal A}
\def\H{\mathcal H}
\def\LM{\mathcal{LM}}

\def\g{\mathfrak{g}}
\def\D{\mathcal D}
\def\P{\mathcal P}
\def\Prim{\mathrm{Prim}\,} 
\def\f{\mathbf k}

\begin{document}

\title{Racks as multiplicative graphs}

\author[J. Mostovoy]{Jacob Mostovoy}
\address{Departamento de Matem\'aticas, CINVESTAV-IPN\\ Col. San Pedro Zacatenco, M\'exico, D.F., C.P.\ 07360\\ Mexico}
\email{jacob@math.cinvestav.mx}

\begin{abstract}
We interpret augmented racks as a certain kind of multiplicative graphs and show that this point of view is natural for defining rack homology.  We also define the analogue of the group algebra for these objects; in particular, we see how discrete racks give rise to Hopf algebras and Lie algebras in the Loday-Pirashvili category $\LM$.  Finally, we discuss the integration of Lie algebras in $\LM$ in the context of multiplicative graphs and augmented racks.
\end{abstract}



\maketitle

Racks are self-distributive algebraic structures which arise in different contexts, such as knot theory or the structure theory of Hopf algebras. They have been invented several times under different names: \emph{wracks} (which later morphed into \emph{racks}), \emph{distributive groupoids}, \emph{automorphic sets}. A union of conjugacy classes in a group can be considered as a rack; there are many other interesting examples. A closely related algebraic structure is that of an \emph{augmented rack} or a \emph{crossed $G$-set}: this is a rack together with a morphism into a group (which is thought of as a rack with the operation of conjugation). We refer to \cite{FR} for an overview of the subject; we shall use the facts from that paper without explicit reference.

In the present note we interpret augmented racks as multiplicative, rather than self-distributive structures; we call these structures \emph{group-like graphs}. This point of view provides a simple interpretation of rack spaces (and, hence, rack homology) and leads to the definition of an analogue of the group algebra for racks: it is a Hopf algebra in the Loday-Pirashvili category of linear maps. This algebra carries a filtration similar to the filtration by the powers of the augmentation ideal in a group algebra; the associated graded Hopf algebra is the 
universal enveloping algebra of a certain Lie algebra in the Loday-Pirashvili category, which can be described in terms of what we call the \emph{graded coinvariant module} of the corresponding augmented rack. Note that Hopf dialgebras and Leibniz algebras are particular cases of Hopf algebras and Lie algebras in the Loday-Pirashvili category, see \cite{LP}, so our constructions can be translated into the dialgebra language. 

We also consider smooth group-like graphs and indicate how set up a version of Lie theory for them. One can think of a Lie algebra in the Loday-Pirashvili category (in particular, a Leibniz algebra) as of an infinitesimal object corresponding to a group-like graph of a certain kind, namely, a \emph{linear Lie graph}. Every such finite-dimensional Lie algebra can be integrated to a linear Lie graph; this is consistent with the formal integration procedure of \cite{M}. Linear Lie graphs are in one-to-one correspondence with \emph{linear augmented Lie racks}; the exponential map of a Lie algebra to the corresponding Lie group is an example of such rack. On the level of non-augmented racks, this integration procedure has been discussed in \cite{K}, where the connection between Leibniz algebras and racks was made for the first time.

Much of what we do is valid not only for group-like graphs, but for \emph{multiplicative graphs}: if group-like graphs are thought of as being analogous to groups, the multiplicative graphs are analogous to semigroups (and are different from \emph{shelves} of \cite{C}). We give a concrete example of a multiplicative graph which is not group-like; it arises in knot theory when considering knots with double points. A by-product of our constructions is a functor that assigns a differential graded Lie algebra to a Lie algebra in the Loday-Pirashvili category (in particular, to a Leibniz algebra).
 
\section{Multiplicative graphs and augmented racks}

\subsection{Multiplicative graphs}
In this note, unless stated otherwise, by a ``graph'' we shall mean a directed graph, possibly with loops and multiple edges. 
Such a graph $Q$ can be written as a pair of sets (vertices $V$ and arrows $A$) with a pair of maps from arrows to vertices (source $s$ and target $t$):
$$Q=(V\underset{t}{\overset{s}{\leftleftarrows}} A).$$
Recall that the \emph{Cartesian product}  $Q_1\, \square\, Q_2$  of two graphs $Q_1=(V_1 \leftleftarrows A_1)$ and $Q_2=(V_2 \leftleftarrows A_2)$ is a graph on the set of vertices $V_1\times V_2$, the set of arrows $A_1\times V_2 \ \sqcup\ V_1\times A_2$ and the source and target maps
$$s=s_1\times \mathrm{id} \sqcup \mathrm{id} \times s_2$$
and 
$$t=t_1\times \mathrm{id} \sqcup \mathrm{id} \times t_2.$$
 If $Q_1$ and $Q_2$ are thought of 1-dimensional cell complexes with directed 1-cells, the graph $Q_1\, \square\, Q_2$ is obtained from the 2-dimensional cell complex $Q_1\times Q_2$ by erasing the 2-cells. The Cartesian product of graphs is associative in the sense that there is a canonical isomorphism $$(Q_1\, \square\, Q_2)\, \square\, Q_3\ \simeq \ Q_1\, \square\, (Q_2\, \square\, Q_3).$$

\begin{defin} A \emph{Cartesian-multiplicative} or, simply, \emph{multiplicative} graph is a graph $Q$ together with a morphism $$\mu: Q\, \square\, Q\to Q$$ which is associative in the sense that $$\mu \circ (\mu\, \square\, \mathrm{id}) = \mu \circ (\mathrm{id}\, \square\, \mu).$$ 
The set of vertices of a multiplicative graph $Q$ is a semigroup. Call $Q$ \emph{group-like} if this semigroup is, actually, a group.  
\end{defin}

As an algebraic structure, a group-like graph is equivalent to what is known as an \emph{augmented rack}. Recall that an augmented rack $X$ over a group $G$ (also called a \emph{crossed $G$-set}) is a $G$-set $X$ together with a morphism (\emph{augmentation map}) of $G$-sets $\pi: X\to G$, where $G$ acts on itself by conjugation. Probably, the most basic example of an augmented rack is a union of an arbitrary set of conjugacy classes in a group $G$. 
 
\medskip 
 
A group-like graph $Q=(G \leftleftarrows A)$  gives rise to an augmented rack in the following fashion.
The ``multiplication'' $\mu$ gives a two-sided action of the group $G$ of vertices of $Q$ on the set $A$ of arrows. 
Let $X\subset A$ be the set of arrows whose source is $1$. Then there is an action of $G$ on $X$  defined by 
$$x^g=g^{-1}\cdot x\cdot g$$
and the target map $t: X\to G$ is a morphism of $G$-actions where $G$ acts on itself by conjugation. In other words, $X$ is an augmented rack over $G$.

Conversely, given an augmented rack $\pi: X\to G$ one constructs a group-like graph $$(G \leftleftarrows G\times X)$$ by setting
$$s(g,x) = g, \quad t(g,x) = g \pi(x),$$
and
\begin{equation}\label{graphtorack}
\mu((g_1,x_1), g_2) = (g_1g_2, x_1^{g_2}), \qquad \mu(g_1, (g_2, x_2)) = (g_1g_2, x_2).
\end{equation}

Group-like graphs form a category, whose morphisms are the morphisms of graphs that respect the multiplication; in particular, they are group homomorphisms on vertices. Augmented racks also form a category whose morphisms are the commutative squares of the form
$$\begin{CD}
X @>f_X>> X'\\
@VV\pi V @VV\pi'V\\
G @>  f_G>>G' ,
\end{CD}$$
where $f_G$ is a homomorphism and $f_X(x)^{f_G(g)} = f_X(x^g).$ 

\medskip

The preceding construction that allows to pass  from a group-like graph to a rack and vice versa is functorial:
\begin{thm}\label{main}
The category of group-like graphs is equivalent to the category of augmented racks.
\end{thm}

There are variations on the definition of a multiplicative graph. One may also consider usual graphs, that is, undirected graphs without loops and multiple edges. These lead to a very particular class of racks. Namely, in this context, for a group-like graph $Q$ with the group of vertices $G$, the set of edges $X$ one of whose ends is the unit of $G$ is a subset of $G$ which is (a) invariant under conjugation; (b) contains  together with each $g\in G$ its inverse. Conversely, for each union $X$ of conjugacy classes in $G$ satisfying (b) there exists a group-like graph whose set of vertices is $G$ and whose set of edges that connect to the unit is $X$.

\medskip

Another variation consists in considering \emph{$k$-graphs}\footnote{This terminology is somewhat arbitrary.} instead of graphs, that is, pairs of sets with $k+1$ maps between them, for any $k\geq -1$. One defines the Cartesian product of $k$-graphs in the same fashion. For instance, when $k=0$, with this definition each group-like 0-graph is of the form  $(G\leftarrow G\times X)$ where $X$ is a $G$-set and the map is the projection onto the first factor. A group-like 0-graph can be considered as a group-like 1-graph whose only arrows are loops.




\subsection{Racks without augmentation and their associated groups} Historically, augmented racks are secondary objects as compared to racks.
A \emph{rack} (without an augmentation) is a set $X$ with a binary operation
$X\times X\to X,$ written as
$$(x,y)\to x^y$$
satisfying the identity
$$ (x^y)^z = {\left(x^z\right)}^{(y^z)}$$
and such that for every $y,z\in X$ there exists a unique $x\in X$ with
$$x^y = z.$$
This definition axiomatized the properties of conjugation in a group: if $X$  is a union of conjugacy classes in a group, the rack structure on $X$ is given by 
$$x^y = y^{-1}xy.$$
The rack structure on a group $G$ given by conjugation is called the \emph{conjugation rack of $G$}.

\medskip
In an augmented rack $\pi: X\to G$, the set $X$ naturally has the structure of a rack with the operation defined as 
$$x^y=x^{\pi(y)}.$$
Conversely, each rack $X$ defines the \emph{associated group} $G_X$ with the presentation
$$\langle \tau_{{x}_1},\ldots, \tau_{{x}_n}\, |\, \tau_{{x}_i} \tau_{{x}_j} = \tau_{x_j} \tau_{\left({x_i}^{x_j}\right)}\rangle,$$ 
where $X=\{x_1, x_2, \ldots \}$.
The group $G_X$ acts on $X$ and the tautological map $\tau: X\to G_X$ which sends  $x\in X$ to the corresponding generator $\tau_{x}$ of $G_X$ is $G_X$-equivariant; that is, to say, defines an augmented rack. The functor $X\to G_X$ is adjoint to the functor of taking the conjugation rack of a group.

\begin{prop}
The group-like graph corresponding to the augmented rack $X\to G_X$ is path-connected.
\end{prop}
\begin{proof}
By construction, $X$ can be identified with the set of edges of the graph that emanate from $1\in G_X$; in particular, each generator $\tau_{x} \in G_X$ is connected to 1 by the edge $x$. Also, the inverse of a generator $\tau_{x}^{-1}$ is connected to 1 by the edge $\tau_{x}^{-1}\cdot x$. Now, If $w\in G_X$ is connected to 1 by a sequence of edges $a_1, \ldots a_n$, with $s(a_1)=1$ and $t(a_n)=w$, then $w \tau_{x}$ is also connected to 1, by the sequence of edges $a_1, \ldots a_n, w\cdot x$, and $w\tau_{x}^{-1}$ -- by the sequence $a_1, \ldots a_n, w \tau_{x}^{-1}\cdot x$.
\end{proof}

\medskip

Note that the only property of the rack $X\to G_X$ that we used in the above proof is that the image of $X$ in $G_X$ generates the whole group $G_X$. A similar argument gives the following statement:
\begin{prop}
For an augmented rack $X\to G$, let $H$ be the subgroup of $G$ generated by the image of $X$. Then $H$ is normal in $G$; it is the group of vertices of the connected component of the unit in the group-like graph of  $X\to G$.
\end{prop}
As we shall see later, the connectivity of a group-like graph is reflected in the properties of the bialgebras associated with it.

\section{The associated spaces}

\subsection{The cubical complexes $EQ$ and $BQ$}
Let $\mathrm{Q}_1$ be the graph with two vertices and one arrow connecting them; set $\mathrm{Q}_n = \mathrm{Q}_{1}\, \square\,\ldots \, \square\, \mathrm{Q}_1$ ($n$ factors). A \emph{$k$-face} of $\mathrm{Q}_n$ is a subgraph of $\mathrm{Q}_n$ isomorphic to $\mathrm{Q}_k$, obtained  by replacing $n-k$ of the copies of $\mathrm{Q}_{1}$ in the product above by one of its vertices.

For a graph $Q$, we shall refer to the morphisms $\mathrm{Q}_n\to Q$ as the \emph{$n$-cubes of $Q$}. 
The restriction of an $n$-cube to one of the $k$-faces of $\mathrm{Q}_n$ is a \emph{$k$-face} of the $n$-cube.
Each arrow $a$ of $Q$ gives rise to a canonical map $a: \mathrm{Q}_1 \to Q$. In a multiplicative $Q$, a \emph{product $n$-cube} is an $n$-cube of the form
$$\mathrm{Q}_n = \mathrm{Q}_1\,\square\ldots\square\, \mathrm{Q}_1\xrightarrow{a_1\,\square\, \ldots\,\square\, a_n} Q\,\square\ldots\square\, Q \overset{\mu_n}{\longrightarrow} Q,$$ 
where $a_i$ are arrows of $Q$ and the last map is the multiplication of $n$ factors in $Q$; we shall denote such $n$-cube by $a_1\,\square\, \ldots\,\square\, a_n$. A  \emph{product square} is a product 2-cube.

\begin{lem}
The $k$-faces of a product $n$-cube are product $k$-cubes. 
\end{lem}
\begin{proof}
If $\mathrm{Q}_0$ is the graph consisting of one vertex and no arrows, we have $\mathrm{Q}_0 \,\square\, \mathrm{Q}_1 =  \mathrm{Q}_1$. An $(n-1)$-face of an $n$-cube is obtained by replacing a copy of $\mathrm{Q}_1$ in the product by a  $\mathrm{Q}_0$ which can be grouped with the copy of $\mathrm{Q}_1$ that precedes or follows it; this establishes the lemma.
\end{proof}

For a multiplicative graph $Q=(G \leftleftarrows A)$ let $EQ$ be the cubical complex\footnote{By a ``cubical complex'' we mean a geometric realization of a precubical set.} whose $n$-dimensional faces correspond to the product $n$-cubes of $Q$ and the face maps -- to the inclusions of the $(n-1)$-faces  into $\mathrm{Q}_n$. The semigroup $G$ of vertices of $Q$ acts on the product $n$-cubes by  
$$(g,   a_1 \,\square\, \ldots \,\square\, a_n) \mapsto (g\cdot a_1) \,\square\, \ldots \,\square\, a_n$$
for $g\in G$ acting on $A$ the left via $\mu$. This action is compatible with the face maps and, hence, descends to an action of $G$ on $EQ$. 
We shall denote the orbit space of this action by $BQ$.

Given a product $k$-cube $a =  a_1 \,\square\, \ldots \,\square\, a_k$ and a product $m$-cube $b = b_1 \,\square\, \ldots \,\square\, b_m$ , one defines the product $(k+m)$-cube $a\,\square\, b$ as $a_1 \,\square\, \ldots \,\square\, a_k  \,\square\, b_1 \,\square\, \ldots \,\square\, b_m$. This operation on product cubes is compatible with the face maps so that there is an associative product 
$$EQ\times EQ\to EQ$$
given by $(a,b)\mapsto a\,\square\, b$. 

\medskip

For group-like graphs, the product cubes are easy to describe:

\begin{lem}\label{prod}
Let $Q=(G \leftleftarrows A)$ be a group-like graph associated with the rack $X$. Product $n$-cubes in  $Q$ are in one-to-one correspondence with $(n+1)$-tuples of the form
$(g, x_1,\ldots, x_n)$ with $g\in G$ and $x_i\in X$. 
\end{lem}

\begin{proof} Recall that $A=G\times X$. We shall see that each product $n$-cube can be uniquely written as 
$$(g, x_1)\,\square\, (1, x_2) \,\square\,\ldots \,\square\, (1, x_n)$$
with $g\in G$ and $x_i\in X$. 

Let us first consider  product squares. The product square $(g_1, x_1)\,\square\, (g_2, x_2)$ is of the form
$$\begin{CD}
g_1g_2\pi(x_2) @>x_1^{g_2\pi(x_2)}>> g_1\pi(x_1)g_2\pi(x_2)\\
@AAx_2A @AAx_2A\\
g_1g_2 @> x_1^{g_2}>>g_1\pi(x_1)g_2.
\end{CD}$$
Here the corners of the square are vertices of $Q$, the arrows are labelled by the elements of $X$ rather than $G\times X$ since the corresponding element of $G$ is indicated at the source of the arrow. 
Now, assume that we are given arbitrary $h=g_1g_2$, $z_1= x_1^{g_2}$ and $z_2=x_2$. Then, the above product square is of the form
$$\begin{CD}
h\pi(z_2) @>z_1^{\pi(z_2)}>> h\pi(z_1)\pi(z_2)\\
@AAz_2A @AAz_2A\\
h @>  z_1>>h\pi(z_1),
\end{CD}$$
and this establishes the lemma for product squares.

The case of the $n$-cubes with $n>2$ now follows by induction from the associativity of the $\square$ operation.
\end{proof}

In the notations of Lemma~\ref{prod}, the action of $G$ on $EQ$ is of the form
$$(g, (g', x_1, \ldots, x_n)) \mapsto (gg', x_1, \ldots, x_n);$$
in particular, it is free and $EQ$ is a covering space of $BQ$.

\begin{lem}
Let $Q$ be a group-like graph associated with the rack $X$. The space $BQ$ coincides with the rack space of $X$.
\end{lem}

The proof is immediate and follows from the definition of the rack space \cite{FRS} and the description of the product squares given in the proof of Lemma~\ref{prod}.

\subsection{The spaces $EQ$ and $BQ$ via the James reduced product}
The cubical complex $EQ$ can be defined in terms of generators and relations. Consider $Q$ as a topological space (one-dimensional simplicial or cubical complex) with a two-sided continuous $G$-action; the group $G$ is then the 0-skeleton of $Q$.  Write $F(Q)$ for the free monoid generated by the points of $Q$ whose identity is the vertex corresponding to $1\in G$. The monoid $F(Q)$ carries the natural topology induced by the topology on $Q$; it is known as the \emph{James reduced product} of $Q$, see \cite{J}. 

Consider the following set of equivalence relations on $F(Q)$:
\begin{equation}\label{rel}
g * x \sim g\cdot x, \quad x* g \sim x\cdot g,
\end{equation}
for each $g\in G, x\in Q$, where $*$ is the product in $F(Q)$ and $\cdot$ denotes the $G$-action on $Q$.
\begin{prop}
For a group-like graph $Q$, the monoid $EQ$ is the quotient of $F(Q)$ by the above equivalence relations.
\end{prop}
\begin{proof}
Denote the congruences (\ref{rel}) by $R$. The space  $F(Q)/R$ has the natural structure of a cubical complex. Indeed, $F(Q)$ is the union of cells of the form
$q_1 *\ldots *q_m$ where each $q_i$ is either a fixed element of $G$ or varies over a fixed arrow in $A$. 
The congruences $R$ respect this cell subdivision of $F(Q)$; each cell in $F(Q)/R$ can be written as $a_1 *\ldots *a_m$ with $a_i\in A$. Moreover, modulo $R$ each of these cells can be uniquely written as $g* x_1*\ldots *x_n$ where the $x_i$ are edges emanating from 1.

The space $EQ$ contains a copy of  $Q$ so that there is a unique continuous surjective homomorphism of $F(Q)$ to $EQ$ which descends to a map 
 $$F(Q)/R\to EQ$$
that maps $a_1*\ldots * a_n$ to $a_1\,\square\ldots\square\, a_n.$ The preimage of a cell $(g, x_1,\ldots, x_n)$ is precisely the cell  $g* x_1*\ldots *x_n$, so that this homomorphism is, actually, an isomorphism.
\end{proof}
\begin{rem}
This construction of the rack space makes sense in a somewhat more general context. Let $Y$ be a topological space and $G\subset Y$ a subset with the group structure such that there is a left\footnote{The same construction works for right and for two-sided actions.} action of $G$ on $Y$ extending the left multiplication on $G$. Denote by $E(Y, G)$ the quotient of the free monoid $F(Y)$ by the relations
$$g * x \sim g\cdot x$$
for all $g\in G, x\in Y$. There is a left action of $G$ on $E(Y, G)$ and we can define $B(Y, G)$ as the quotient space of $E(Y, G)$ by this action. For instance, given a left $G$-set $X$, one can take $Y= G\sqcup X$. A possibly more interesting example is a Cayley graph for $G$ or, indeed, a graph determined by an arbitrary subset $S\subseteq G$: we set $Y$ to be the graph whose vertices are elements of $G$ and whose edges are pairs $(g, gs)$ for all $g\in G$ and $s\in S$. Note that the 1-skeleton of $E(Y, G)$ in this case is a multiplicative graph 
into which the Cayley graph is embedded.
\end{rem}

\subsection{The action of $\pi_1 BQ$ on $\pi_n EQ$}
The fundamental group of any topological monoid (or, indeed, of any space with a unital multiplication) is abelian. Moreover, for any covering $p: E\to B$ such that $E$ is a topological monoid $\pi_1 E$ lies in the centre of $\pi_1 B$; in particular, $\pi_1 EQ$ lies in the centre of $\pi_1 BQ$.  

The proof is a standard exercise in topology. Let $p(1)$ be the basepoint in $B$; for a curve $\gamma$ starting at $p(1)$ write $\overline{\gamma}$ for its lifting to $E$ with $\overline{\gamma}(0)=1$. Let $\alpha,\beta:[0,1]\to B$ be two closed paths starting and ending at $p(1)$ and assume that $\overline{\alpha}$ is closed in $E$. Denote by $\overline{\alpha}_\tau$ and $\overline{\beta}_\tau$ the reparametrizations of $\overline{\alpha}$ and $\overline{\beta}$, respectively, which are constant outside of the interval $[\tau,\tau+1/2]$. Note that the pointwise product curves $\overline{\alpha}_{\tau_1} \overline{\beta}_{\tau_2}$ are fixed-end homotopic in $E$ for all $\tau_1,\tau_2\in [0,1/2]$; since $\overline{\alpha}$ is closed, their projections to $B$ define closed loops in $B$. Now, writing $\circ$ for the concatenation of loops in $B$ we see that 
$$\overline{\alpha\circ\beta} =\overline{\alpha}_0 \overline{\beta}_{1/2},$$
while
$$\overline{\beta\circ\alpha} =\overline{\beta}_0 \overline{\alpha}_{1/2} \sim \overline{\beta}_{1/2} \overline{\alpha}_{0} = \overline{\alpha}_0 \overline{\beta}_{1/2},$$
which means that $\alpha\circ\beta$ equals $\beta\circ\alpha$ as an element of $\pi_1 B$.

A very similar argument shows that $\pi_1 B$ acts trivially on $\pi_n E$ for all $n$. Since for $n>1$ the groups $\pi_n E$ coincide with $\pi_n B$, we see that the rack space $BQ$ is always homotopy simple, the fact that was first proved in \cite{FRS}.

\section{Examples}

\subsection{Path graphs}\label{pathgraphs} Probably, the simplest non-trivial example of a multiplicative graph is the graph of pairs of elements in a (semi)group. For a semigroup $G$, let $A=G\times G$ with the maps $s$ and $t$ being the projections onto the first and on the second factors respectively, and the two-sided action of $G$ being the action of the diagonal in $G\times G$ by multiplication. When $G$ is group-like, the corresponding augmented rack is the identity map $G\to G$ (and the non-augmented rack is the conjugation rack of $G$).

An extension of this example is the graph whose arrows are sequences of $n$ elements in $G$, with $s$ and $t$ being the first and the last element of the sequence. When $G$ is a topological semigroup, one can consider the graph whose edges are directed paths on $G$; again, with $s$ and $t$ being the beginning and the end of the path.  
If $G$ is a Lie group, it makes sense to consider the group-like graph $LG$ whose arrows are directed segments of geodesics in $G$, parametrized by length. An important feature of this graph is that the maps $$s,t: LG\to G$$ are vector bundles such that the action of $G$ on $LG$ is linear on the fibres. Indeed, each geodesic $\gamma$ starting at $g\in G$ is determined by a tangent vector in $T_g G$ whose direction coincides with that of $\gamma$ and whose length gives the length of $\gamma$, and, therefore, the fibre of $s$ over $g\in G$ can be identified with $T_g G$; the same argument goes for the map $t$. This is the most basic example of a \emph{linear Lie graph}, see Section~\ref{LLG}.



\subsection{1-skeleta of multiplicative cubical complexes and of simplicial monoids}
For any cubical complex with an associative cubical multiplication on it, the 1-skeleton is a multiplicative graph. Each multiplicative graph $Q$ can be obtained in such a way, being the 1-skeleton of the cubical complex $EQ$. Here, as before, a cubical complex is a geometric realization of a precubical set $X$ (a ``cubical set without degenerations") and a cubical multiplication is a geometric realization of an associative map $X\times X\to X$. Recall that the $n$-cubes of the product of two precubical sets $X$ and $Y$ are pairs of cubes of $X$ and $Y$ whose dimensions sum up to $n$.

Also, the 1-skeleton of a simplicial monoid is a multiplicative graph whose vertices are the 0-simplices, whose edges are the 1-simplices and whose source and target maps are the face maps. The monoid of 0-simplices can be identified with the monoid of the degenerate 1-simplices, and, therefore, acts on the 1-simplices by multiplication; since the face maps are homomorphisms, this gives a multiplicative graph. When the monoid of 0-simplices is, actually, a group, its action respects the non-degeneracy of 1-simplices. In particular, in this case, one obtains a multiplicative graph whose vertices are the 0-simplices of the simplicial monoid and whose arrows are the {non-degenerate} 1-simplices.

\subsection{Knot racks as multiplicative graphs}

One of the most useful examples of racks is the rack associated with a framed knot. In terms of multiplicative graphs, this construction has the following form.

Let $K$ be a parametrized framed knot in $\mathbb{R}^3$. Choose a basepoint in the exterior of $K$ and let $A$ be the set of all homotopy classes of smooth loops in $\mathbb{R}^3$ which start at the basepoint and cross the knot $K$ exactly once with the positive sign (this means that at the crossing point the tangent vector to the loop, the tangent vector to the knot and the framing vector form a positive basis in  $\mathbb{R}^3$). Each $a\in A$ gives rise to two elements of the knot group $\pi_1(\mathbb{R}^3 - K)$ as follows. Represent $a$ by a curve $\gamma$ and define $s(a)$ as the class of the loop obtained by moving $\gamma$ at the crossing point off the knot in the direction opposite to the framing; similarly, $t(a)$ is obtained by shifting $\gamma$ off $K$ along the framing. These maps are evidently well-defined. Moreover, there is a two-sided action of $\pi_1(\mathbb{R}^3 - K)$ on $A$ by pre- and post-composing the loops which cross the knot once with the loops that never cross the knot. Therefore $(\pi_1(\mathbb{R}^3 - K) \leftleftarrows A)$ is a group-like graph. 

It is quite clear that this construction indeed is equivalent to the usual knot rack. Indeed, if $s(a)$ is trivial, a loop $\gamma$ representing $a$ bounds a disk in the exterior of $K$. This disk can be squeezed, without moving the crossing point of $\gamma$ with $K$, to an interval connecting the basepoint with a point on the knot; this is how the knot rack is normally defined. 

The space $EQ$ for a knot rack also has a natural definition in terms of the curves which cross the knot. Namely, the set of homotopy classes of loops that intersect the knot $n$ times, each of them positively, coincides with the set of product $n$-cubes in the corresponding group-like graph. Indeed, each loop that crosses the knot $n$ times is a concatenation of $n$ loops that hit the knot exactly once. Such a decomposition is not unique; however, the product cube made of these $n$ loops is well-defined.

We should also note that, just as in the case of usual knot racks, here framed knots in $\mathbb{R}^3$ can be replaced by $n-2$-dimensional framed submanifolds of an $n$-manifold, or even by $n-k$-dimensional submanifolds with $k>2$. In this latter case, the fundamental group of the knot exterior should be replaced by its $k-1$st homotopy group. Probably, the most basic example of a graph defined by a codimension two subset is the group-like graph associated with the complement of an $n$-point set in $\mathbb{R}^2$. Its vertices are the elements of the free group on $n$ generators $x_1,\ldots , x_n$ and the arrows are in one-to-one correspondence with pairs of elements of the form $(ab, a x_k b)$ for some $x_k$. The rack that corresponds to such a graph is the union of the conjugacy classes of the generators in the free group.

\subsection{String links with double points}

An interesting example of a multiplicative graph which is not group-like is provided by string links with one double point.

Given $n>0$, denote by $L_n$ the monoid of isotopy classes of string links on $n$ strands and let $L_n^{\bullet}$ be the set of isotopy classes of string links with one transversal double point (see \cite{CDbook} for the definitions). A string link with a double point can be composed with a usual string link on either side and this gives two commuting actions 
$$L_n\times L_n^{\bullet} \to L_n^{\bullet}$$
and
$$ L_n^{\bullet}\times L_n \to L_n^{\bullet}.$$
There are also two maps  $L_n^{\bullet}\to L_n$ given by the positive and the negative resolution of the double point according to the Vassiliev skein relation; these two maps are compatible with the actions of $L_n$ on $L_n^{\bullet}$. This means that $\mathcal{L}_n:= (L_n \leftleftarrows L_n^{\bullet})$ is a multiplicative graph.

The corresponding space $E\mathcal{L}_n$ has been mentioned in the literature; namely, it was constructed by Matveev and Polyak  in \cite[page 229]{MaPo}.

\medskip

The graph $\mathcal{L}_n$ has a group-like subgraph which consists of pure braids with one double point. Two pure braids involved in the resolution of a double point differ by an insertion of a generator: if the positive resolution can be written as $ab$, the negative is of the form $aA_{ij}b$, where $i$ and $j$ are the numbers of the strands which cross at the double point\footnote{There is a freedom of choosing the sign in the Vassiliev skein relation for string links which involves the relative order in each pair of strands; this is the same as choosing between a generator and its inverse for each $i,j$.}. The corresponding rack is the union of the conjugacy classes of the generators in the pure braid group.

\section{Hopf algebras associated with multiplicative graphs}
Any group $G$ gives rise to two Hopf algebras. One is the group algebra $\f[G]$; it carries the filtration by the powers of the augmentation ideal $I(G)$. The associated graded Hopf algebra $\D(G)$ is the universal enveloping algebra of the graded Lie algebra coming form the lower central series of $G$.
In this section we construct the analogs of both Hopf algebras for group-like graphs. 

\subsection{The Loday-Pirashvili category}\label{LPC} Let us review some notions from  \cite{LP}. The \emph{Loday-Pirashvili category} $\LM$ over a given field $\f$ (which we assume to be of characteristic zero) has, as objects, pairs of vector spaces $$\bigl(U\xrightarrow{\ f}V\bigr)$$ over $\f$ together with a linear map between them.  The morphisms are the commutative squares of the form 
$$\begin{CD}
U @>h_1>> U'\\
@VVfV @VVf'V\\
V @>  h_0>>V'.
\end{CD}$$
The category $\LM$ is a tensor category with the tensor product\footnote{Hereafter we denote by $`+"$ the direct sum of vector spaces.}
$$\bigl(U\xrightarrow{\ f}V\bigr)\otimes \bigl(U'\xrightarrow{\ f'}V'\bigr) = \bigl(U \otimes V'+V\otimes U'  \xrightarrow{\ f\otimes \mathrm{Id} + \mathrm{Id} \otimes f'} V\otimes V' \bigr).$$
The natural explanation for this tensor product comes from considering $\LM$ as the category of $1$-jets of differential graded vector spaces, that is, the quotient of the category of (non-negatively graded) chain complexes over $\f$ by the subcategory of the chain complexes with trivial components in degrees 0 and 1.

One may speak of Lie algebras and Hopf algebras in $\LM$; the Milnor-Moore and the Poincar\'e-Birkhoff-Witt theorems are then valid in $\LM$. 
An algebra in $\LM$ is a pair 
\begin{equation}\label{alg}
(\A\xrightarrow{f}\H),
\end{equation}
with $\H$ an algebra, $\A$ an $\H$-bimodule and $f$ a bimodule map. A bialgebra in $\LM$ has, in addition, a compatible dual structure, namely a coproduct $\Delta_0:\H\to \H \otimes \H$
which makes $\H$ into a bialgebra, and the two-sided coaction 
$$\Delta_1: \A\to \A\otimes \H + \H\otimes\A,$$ 
which is an $\H$-bimodule map satisfying
$$\Delta_0\circ f = (f \otimes \mathrm{Id} + \mathrm{Id}\otimes f)\circ \Delta_1.$$
A bialgebra is a Hopf algebra if it has an antipode; we shall state the properties satisfied by the antipode  in Section~\ref{mga}.

A Lie algebra in $\LM$ is a pair 
\begin{equation}\label{liealg}
(M\xrightarrow{f}\g),
\end{equation}
where $\g$ is a Lie algebra, $M$ is a right\footnote{We take right, rather than left, modules so as to have the same conventions as \cite{LP}.} module over $\g$ and $f$ is a $\g$-module morphism. In this situation, there exists a bracket on $M$ that gives it the structure of a Leibniz algebra:
$$[x,y]= [x, f(y)],$$
where the bracket on the right-hand side denotes the right action of $\g$ on $M$. Each Leibniz algebra $L$ can be obtained this way by taking the Lie algebra $\g$ to be the maximal antisymmetric quotient of $L$ and the map $f$ to be the corresponding quotient map. We refer to \cite{LP} for more details.

If $M$, considered as a Leibniz algebra, is, actually, a Lie algebra on which $\g$ acts by derivations, the Lie algebra $(M,\g)$ is a \emph{differential crossed module}. Differential crossed modules can be identified with the differential graded Lie algebras whose only nonzero terms are in degrees 0 and 1.

\subsection{Lie algebras in $\LM$ and differential graded Lie algebras}\label{dgla} 
There is one important observation that we should make even though it will not be used in what follows:
the Lie algebras in $\LM$ are related to the differential graded Lie algebras in the same way as the multiplicative graphs are related to the products on cubical complexes.

Namely, given a differential graded Lie algebra $\g_*$:
$$\ldots\xrightarrow{d} \g_1 \xrightarrow{d} \g_0,$$ 
the map $\g_1 \xrightarrow{d} \g_0$ together with the restriction of the bracket of $\g_*$ to the maps $\g_1\otimes\g_0\to \g_1$ and  $\g_0\otimes\g_0\to \g_0$ is a Lie algebra in $\LM$.

Conversely, given a Lie algebra $\g= \bigl(\g_1 \xrightarrow{d} \g_0\bigr)$ in $\LM$ with the Lie bracket denoted by $[\,\cdot\,,\, \cdot\,]$, consider the free Lie algebra on the vector space $\g_0+\g_1$. It is a differential graded Lie algebra with the differential induced by $d$; we denote its Lie bracket by $\llbracket\,\cdot\,,\, \cdot\,\rrbracket$. Let $E\g$ be the quotient of this free differential graded Lie algebra by the relations
\begin{equation}\label{relations}
\llbracket x, y \rrbracket = [x,y]
\end{equation}
where either $x,y\in \g_0$, or $x\in \g_1$ and $y\in \g_0$.
\begin{prop}
The differential graded Lie algebra $E\g$, considered modulo terms of degree two and higher, coincides with $\g$ as a Lie algebra in $\LM$.
\end{prop}
The proof is immediate.
Observe that the functor $\g\mapsto E\g$ from $\LM$ to $\mathrm{DGLA}$ is the left adjoint to the functor of reduction modulo terms of degree 2 and higher.
\subsection{The multiplicative graph algebra}\label{mga}

Let $Q=(G \leftleftarrows A)$ be a multiplicative graph. Write $\f[A]$ for the $\f$-vector space spanned by $A$ and consider the linear map
$$\phi:\f[A]\to \f[G]$$
to the semigroup algebra of $G$ defined by
$$\phi(a) = t(a)-s(a)$$
for all $a\in A$. Then $\bigl(\f[A]\xrightarrow{\phi} \f[G]\bigr)$ is an algebra in the Loday-Pirashvili category $\LM$; the two-sided action of $\f[G]$ on $\f[A]$ is the linear extension of the two-sided action of $G$ on $A$.

Recall that the algebra $\f[G]$ carries a coproduct, which we denote by $\Delta_0$:
$$\Delta_0(g) = g\otimes g$$
for all $g\in G$.
Define the two-sided coaction $\Delta_1$ 
by setting
$$\Delta_1(a) = a\otimes t(a) + s(a) \otimes a$$
for all $a\in A$. Note that $\phi$ sends $\Delta_1$ to  $\Delta_0$; 
moreover, $\Delta_1$ is a $\f[G]$-bimodule map. Indeed, for $a\in A$ and $g \in G$ we have
$$
\Delta_1 (a\cdot g) = a\cdot g \otimes t(a\cdot g) + s(a\cdot g)\otimes a\cdot g  =  a\cdot g \otimes t(a) g  + s(a) g\otimes a\cdot g= \Delta_1(a)\cdot \Delta_0(g).
$$
Similarly, $\Delta_1 (g\cdot a) =  \Delta_0(g)\cdot \Delta_1 (a)$.
As a consequence, we have
\begin{lem}
$\bigl(\f[A]\xrightarrow{\phi} \f[G]\bigr)$ is a bialgebra in $\LM$.  
\end{lem}

When $G$ is group-like, we can define the involution $S_1: \f[A]\to \f[A]$ as
$$a \mapsto -s(a)^{-1}\cdot a\cdot t(a)^{-1}$$
for each $a\in A$.
It maps under $\phi$ to the antipode $S_0: \f[G]\to \f[G]$ that sends each group element $g$ to $g^{-1}$. In order to check that $S_1$ gives rise to an antipode in $\LM$ we need to verify that for any $a\in A$
\begin{equation}\label{antipode}
\mu\circ (S\otimes \mathrm{Id})_1 \circ \Delta_1 = \mu\circ (\mathrm{Id} \otimes S)_1 \circ \Delta_1 = 0,
\end{equation}
Here $\mu$ is the product, $(S\otimes \mathrm{Id})_1$ stands for $$S_1\otimes \mathrm{Id}_{\f[G]} + S_0\otimes \mathrm{Id}_{\f[A]}$$ and $(\mathrm{Id} \otimes S)_1$ for $$ \mathrm{Id}_{\f[A]}\otimes S_0+  \mathrm{Id}_{\f[G]}\otimes S_1,$$ respectively. Substituting these expressions into (\ref{antipode}), we obtain
$$\left(-s(a)^{-1}\cdot a \cdot t(a)^{-1}\right)t(a)+s(a)^{-1}\cdot a = 0$$
and 
$$a\cdot t(a)^{-1} + s(a) \cdot \left(-s(a)^{-1}\cdot a\cdot t(a)^{-1}\right)=0.$$
We have proved
\begin{prop}
When $Q$ is group-like, $\bigl(\f[A]\xrightarrow{\phi} \f[G]\bigr)$ is a Hopf algebra in $\LM$.  
\end{prop}
Note that, in general, this Hopf algebra is not cocommutative.

\subsection{The augmentation filtration and the associated Hopf algebra}
Write $I(G)$ for the augmentation ideal in $\f[G]$ and let $I^n (A) \subset \f[A]$ be the subspace 
$$\langle v_1\cdot a\cdot v_2\, |\, a\in A,\, v_1\in I^k(G),\, v_2\in I^{n-k}(G),\, k\leq n\rangle.$$ 
The image of $I^n (A)$ under $\phi$ lies in $I^{n+1}(G)$.
\begin{lem}\label{connected}
If $Q$ is group-like and path-connected, $\phi(I^n (A)) = I^{n+1}(G)$. 
\end{lem}
\begin{proof} It is sufficient to show that $\phi(A)=I(G)$. 

The augmentation ideal $I(G)$ is additively generated by the elements of the form $g-1$ with $g\in G$. Since $Q$ is path-connected, for any $g\in G$ there is a path $a_1, \ldots, a_n$, with $a_i\in A$ connecting $1$ and $g$, although not necessarily according to the directions of the $a_i$. Let $e_i=1$ if $a_i$ is directed along the path and $e_i=-1$ otherwise. Then 
$$\phi\left(\sum e_i a_i\right)=g-1,$$
which shows that $\phi(A)$ coincides with $I(G)$.
\end{proof}
\begin{rem}
For a not necessarily path-connected $Q$, the vertices of the connected component that contains the unit in $G$ form a normal subgroup $G_0$. Let $I(G,G_0)$ be the kernel of the homomorphism $\f[G]\to \f[G/G_0]$ induced by the quotient map. Then $\phi(I^n (A))$ can be explicitly identified as $I^{n+1}(G,G_0)$. 
\end{rem}

From now on we shall assume that $Q$ is group-like and path-connected.

The maps $\Delta_1$ and $S_1$ respect the filtration by the $I^n(A)$ so there is a graded Hopf algebra $(\D(Q)\xrightarrow{\phi_*} \D(G))$ associated with it. 
The map $\phi_*$, induced by $\phi$, raises the degree by one.
\begin{lem}
$\bigl(\D(Q)\xrightarrow{\phi_*}\D(G)\bigr)$ is an irreducible cocommutative Hopf algebra. 
\end{lem}
\begin{proof}
We have
$$\Delta_1(a)= a\otimes s(a) + s(a)\otimes a + a\otimes\phi(a).$$
The ``non-cocommutative part'' of $\Delta_1$ 
$$\Delta_1': a\mapsto a\otimes\phi(a)$$
vanishes on the associated graded level since $\phi_*$ raises the degree by one. Indeed, for any $u\in \f[A]$ and any $g\in G$ we have 
$$\Delta_1'(g\cdot u) = (g\otimes g)\cdot \Delta_1'(u),$$
and hence, 
$$\Delta_1'((g-1)\cdot u) = \bigl( (g-1)\otimes (g-1) +1\otimes(g-1) + (g-1)\otimes 1\bigr)\cdot \Delta_1'(u).$$
The same kind of equality holds for $\Delta_1'(u\cdot (g-1))$. Moreover, $\Delta'_1(a)\in A\otimes I(G)$ for any $a\in A$. These formulae show that $\Delta_1'$ increases the filtration index by one and, therefore, induces the zero map on $\D(Q)$. This implies that the coproduct in the bialgebra $(\D(Q)\to \D(G))$ is cocommutative.
The irreducibilty follows from the irreducibilty of $ \D(G)$.
\end{proof}
Since $\D(G)$ satisfies the conditions of the Milnor-Moore Theorem, $(\D(Q)\xrightarrow{\phi_*}\D(G))$ also does (see \cite{LP}); therefore, it is the universal enveloping algebra of a certain Lie algebra $(M\to \g)$  in $\LM$:
$$\bigl(\D(Q)\xrightarrow{\phi_*}\D(G)\bigr)\simeq (U(\g)\otimes M \to U(\g)).$$
The Lie algebra $\g=\Prim{\D(G)}$ is the graded Lie algebra of the successive quotients of the lower central series of $G$, tensored with $\f$. As for $M$, it can be understood in terms of the \emph{graded coinvariant module} of the augmented rack corresponding to $Q$.

\subsection{The coinvariant module of an augmented rack}\label{coinvariant}

Let $\pi: X\to G$ be an augmented rack. The vector space $\f[X]$ spanned by $X$ has a decreasing filtration by the subspaces
$$I^n(X) = \langle x^{(g_1-1)\ldots (g_n-1)}\,|\, x\in X, g_i\in G\,\rangle,$$
where  we use the exponential notation for the linear extension of the action of $G$ to an action of $\f[G]$ on $\f[X]$. We should warn that this notation might be not entirely intuitive; for instance, $x^{(g-1)} = x^g - x$. However, we want to keep clear the distinction between the $G$-action on $X$ and the two $G$-actions in a group-like graph. When $n=0$, we set $I^0(X)=\f[X]$.

Set, for $n>0$
$$\P^{n} (X) : = I^n(X) / I^{n+1}(X);$$
in other words, the space $\P^{n} (X)$ consists of the coinvariants of the $G$-action on ${I^n(X)}$.  In particular, $\P^0(X)=\f [X/G]$ is the vector space spanned by the orbits of the action of $G$ on $X$.  Write $\P(X)$ for the graded vector space whose part of degree $k$ is $\P^k(X)$. 

\medskip

It is clear from the definition that the space $\P(X)$ is a graded module over $\D(G)$; we call it the \emph{coinvariant module}. Moreover, the map $\pi$ induces a degree 1 map of graded $\D(G)$-modules $\pi_*:\P(X)\to\D(G)$: 
$$\bigl( x^{(g_1 -1)\ldots (g_n-1)}\,\mathrm{mod}\ I^{n+1}(X)\bigr)\ \mapsto\ \bigl((\pi(x)-1)^{(g_1 -1)\ldots (g_n-1)}\,\mathrm{mod}\ I^{n+2}(G)\bigr).$$
Indeed, 
for $u\in I^k(G)$ and $g\in G$ we have 
$$u^g-u = g^{-1}(ug-gu)= g^{-1}(u(g-1)-(g-1)u)\in I^{k+1}(G)$$ so that $\pi_*$ is well-defined. The image of $\pi_*$ in $\D(G)$ lies in the Lie algebra $\Prim{\D(G)}$ of the primitive elements of $\D(G)$: $\pi_*\P^0(X)$ consists of elements of degree one in $\D(X)$, which are primitive, and $\pi_*\P^n(X)$ is spanned by the commutators with $\pi_*\P^{n-1}(X)$. This shows that $(\P(X) \xrightarrow{\pi_*} \Prim{\D(G)})$ is a graded Lie algebra in $\LM$.

\begin{prop}
The Hopf algebra $\bigl(\D(Q)\xrightarrow{\phi_*}\D(G)\bigr)$ is the universal enveloping algebra of the Lie algebra $\bigl(\P(X) \xrightarrow{\pi_*}  \Prim{\D(G)}\bigr)$, where $X$ is the augmented rack corresponding to the group-like graph $Q$.
\end{prop}

\begin{proof}
The universal enveloping algebra of $\bigl(\P(X) \xrightarrow{\pi_*}  \Prim{\D(G)}\bigr)$ is the map 
$$\bigl(\D(G) \otimes \P(X)\xrightarrow{\mu(\mathrm{Id}\otimes \pi_*)}\D(G)\bigr),$$ 
where $\mu$ is the product in $\D(G)$, with the following $\D(G)$-bimodule structure on $\D(G) \otimes \P(X)$:
\begin{equation}\label{left}w_1\cdot (w_2\otimes m) = (w_1 w_2) \otimes m,\end{equation}
\begin{equation}\label{right}(w\otimes m)\cdot a  = (w a)\otimes m + w \otimes m^a\end{equation}
for all $w_1, w_2, w\in\D(G)$, $m\in \P(X)$ and $a\in \Prim{\D(G)}$ (see \cite[page 271]{LP}).

\medskip

Identify the vector space $\f[A]$ with $\f[G]\otimes\f[X]$ via $(g,x)\mapsto g\otimes x$. Under this identification, $I^n(A)$ is sent to 
$$\sum_{p+q=n} I^p(G)\otimes {I^{q}(X)},$$
namely,
$$ (g_1-1)\ldots(g_p-1) \cdot (g,x)^{(h_1-1)\ldots(h_q-1)}\ \mapsto\ (g_1-1)\ldots(g_p-1) g\, \otimes\, x^{(h_1-1)\ldots(h_q-1)}.$$ 
Therefore, as a graded vector space, $\D(Q)$ is isomorphic to  $\D(G) \otimes \P(X)$. Under this identification, the map $\phi_*$ coincides with $\mu(\mathrm{Id}\otimes\pi_*)$. Also, the  $\D(G)$-bimodule structure is the same as that of the universal enveloping algebra of $(\P(X) \to \Prim{\D(G)})$. This is clear for the left module structure (\ref{left}). As for the right module structure, the action on the right by $a\in \g$ can be represented by a right action of $h-1$ with $h\in G$:
$$(u\otimes x)^{h-1} = (u\otimes x)^h - u\otimes x = u(h-1)\otimes x^h+ u\otimes x^{h-1},$$
with $u\in I^m(A)$ and $x\in I^n(X)$.  It is, actually, sufficient to consider $a\in\Prim{\D(G)}$ of degree one, since elements of this kind generate $\D(G)$; that is, consider the above equality modulo 
$I^{m+n+2}$. Then, the right-hand side of the above formula is equivalent to $$u(h-1)\otimes x+ u\otimes x^{h-1},$$ which in $\D(Q)$ translates precisely into (\ref{right}).
\end{proof}

\subsection{Edge-like elements in Hopf algebras and the Malcev completion}\label{edgelike}
Let $(\A\xrightarrow{} \H)$ be a Hopf algebra in $\LM$.
Call an element $a\in \A$ \emph{edge-like} if
$$\Delta_1(a) = a\otimes t(a) + s(a)\otimes a,$$
where $s(a)$ and $t(a)$ are group-like elements of $\H$. Let us denote the set of group-like elements of $\H$ by $\mathcal{G}_0(\H)$ and the set of edge-like elements of $\A$ by $\mathcal{G}_1(\A)$. Assigning to an edge-like element $a$ the corresponding group-like elements $s(a)$ and $t(a)$ we define two maps 
$$s,t: \mathcal{G}_1(\A) \to \mathcal{G}_0(\H).$$ 
The two-sided action of $\H$ on $\A$ restricts to the action of the group $\mathcal{G}_0(\H)$ on  $\mathcal{G}_1(\A)$. We have:
\begin{lem}
The pair $(\mathcal{G}_0(\H)\leftleftarrows \mathcal{G}_1(\A))$ is a group-like graph.
\end{lem}
Consider the group-like graph Hopf algebra $\phi: \f[A]\to \f[G]$. Since $\phi$ maps $I^{n}(A)$ to  $I^{n+1}(G)$, it descends to a map
$$\phi_n: \f[A]/I^n(A)\to \f[G]/I^{n+1}(G),$$
which is also a Hopf algebra in $\LM$. 
There is a canonical morphism of $\phi_{n+1}$ onto $\phi_n$ for each $n$; the inverse limit of the $\phi_i$ in $\LM$ is a \emph{complete} Hopf algebra
$$\phi\, { \hat{ }}: \f[A]\, \hat{ }\to \f[G]\, \hat{ },$$
whose edge-like elements form a group-like graph
$$\mathcal{G}_0(\f[G]\, \hat{ }\,)\leftleftarrows \mathcal{G}_1(\f[A]\, \hat{ }\,)$$
that we call the \emph{Malcev completion} of $(G\leftleftarrows A)$. 
Here, a complete Hopf algebra in $\LM$ is defined as usual: the tensor products in the definition of the comultiplication should be replaced by the completed tensor products. The definition of edge-like elements in a complete Hopf algebra also uses completed tensor products instead of the usual tensor products.




\section{Linear Lie graphs and linear augmented racks}

\subsection{Lie theory}\label{LLG} One can consider multiplicative graphs such that both $G$ and $A$ are smooth manifolds and $s$ and $t$ are submersions onto their image\footnote{This should be compared with the definition of a Lie groupoid.}. Our principal motivation here is to show how the Lie theory for multiplicative graphs is related to the Lie algebras in $\LM$;  for these purposes it is sufficient to consider a narrower class of graphs.
\begin{defin} A group-like graph $(G\leftleftarrows A)$ over a Lie group $G$ is called a \emph{linear Lie  graph} if the source and target maps are vector bundles over their image, the two-sided action  of $G$ on $A$ is linear on the fibres, and the fibres of $s$ and $t$ over $1\in G$ have the same origin $e\in A$. 
\end{defin}
\begin{defin}
An augmented rack $\pi: X\to G$ is called a  \emph{linear augmented Lie rack} if $X$ is a vector space, $G$ is a Lie group, $\pi$ is smooth with $\pi(0)=1$ and the action of $G$ on $X$ is linear. 
\end{defin}

A linear Lie graph clearly gives rise to a linear augmented Lie rack. The converse is also true. Indeed, the source map $G\times X\to G$ is simply the projection onto the first factor. The target map sends $(g,x)$ to $g\pi(x)$; its fibre over $h\in G$ is the subspace $$\{ (h\pi(x)^{-1}, x)\, |\, x\in X\}.$$ If the point $x\in X$ is taken to be the parameter for the fibre, the left action of $G$ is trivial and the right action is the rack action; both are linear.

\medskip

An example of a linear Lie graph was given in Section~\ref{pathgraphs}:  the arrows of this graph are directed segments of geodesics on a Lie group, parametrised by length. The corresponding augmented rack is the exponential map of the Lie algebra to the Lie group.

\medskip

In a linear augmented Lie rack, the map $\pi: X\to G$ induces a $G$-equivariant\footnote{Here, in order to be consistent with the choice of the definition for the Lie algebra in $\LM$ one has to consider the \emph{right} action of $G$; accordingly, the adjoint representation should be the \emph{right} adjoint representation.} map of the tangent spaces
$$\pi_*: X = T_e X \to T_1 G = \g.$$
Considering the infinitesimal part of the $G$-action, we see that $\pi_*$ is a map of $\g$-modules and, hence defines a Lie algebra in $\LM$. 
Conversely, a Lie algebra in $\LM$, that is, a homomorphism $\g\to \mathfrak{gl}(X)$ which covers the adjoint representation of $\g$ by means of a map $$f:X\to\mathfrak{g},$$
for finite-dimensional $\g$ can be integrated so as to produce a morphism $F$ from a $G$-action on $X$ to the adjoint representation of $G$. Then,  the composition
$$X \xrightarrow{F} \mathfrak{g}\xrightarrow{\exp} G$$
defines a linear augmented Lie  rack. In terms of racks without augmentation, this exact construction can be found in \cite{K}.


\medskip

Therefore, a Lie algebra in $\LM$  produces a linear Lie graph. This graph can be thought of as the global integration of the Lie algebra; in this picture, the formal integration (as described in  \cite{M}) may be thought of as standing halfway between a Lie algebra in $\LM$ and the corresponding linear Lie graph. 

An augmented rack similar to a linear augmented Lie rack arises from the completion of the graded coinvariant module of an augmented rack. The map of $\D(G)$-modules $\pi_*:\P(X)\to \D(G)$ described in Section~\ref{coinvariant} is a degree 1 map of graded vector spaces and can be extended to the map between the graded completions $\overline{\P(X) }$ and   $\overline{\D(G)}$. The image of $\overline{\P(X)}$ lies in the subspace of primitive elements of $\overline{\D(G)}$ and, therefore, the image of the composition
$$\overline{\P(X)} \xrightarrow{\overline{\pi_*}} \Prim{\overline{\D(G)}}\xrightarrow{\exp} \overline{\D(G)}$$
lies in the group $\mathcal{G}_0(\overline{\D(G)}) $ of the group-like elements of the complete Hopf algebra $\overline{\D(G)}$. In particular, if $\f=\mathbb{R}$, the rack
$\overline{\P(X)}\to \mathcal{G}_0(\overline{\D(G)})$ is a linear augmented Lie rack.

\begin{rem} There is a situation where a Lie algebra in $\LM$ can be integrated to a ``more non-linear'' graph than a linear Lie graph. As mentioned in Section~\ref{LPC}, differential crossed modules (or, which is the same, crossed modules of Lie algebras) are Lie algebras in $\LM$. A differential crossed module can be integrated to a \emph{crossed module of Lie groups} (see \cite{BL}), which is an augmented rack that, in turn, gives rise to a group-like graph. The corresponding linear Lie graph can be recovered by taking the tangent spaces to the fibres of the source map of this graph.
\end{rem}

\end{document}